\documentclass[11pt, reqno]{amsart}
\usepackage{amsfonts,amssymb,amscd,amstext,mathrsfs,color}
\usepackage[dvipsnames]{xcolor}
\usepackage{graphicx}
\usepackage{tikz}
\usepackage{hyperref}
\usepackage{mathtools}

\numberwithin{equation}{section}

\newtheorem{theorem}{Theorem}[section]
\newtheorem{corollary}[theorem]{Corollary}
\newtheorem{lemma}[theorem]{Lemma}
\newtheorem{proposition}[theorem]{Proposition}

\theoremstyle{definition}
\newtheorem{remark}[theorem]{Remark}
\newtheorem{question}{Question}
\newtheorem{definition}[theorem]{Definition}
\newtheorem{example}[theorem]{Example}

\newcommand{\C}{\mathbb{C}}
\newcommand{\D}{\mathbb{D}}
\newcommand{\B}{\mathbb{B}}

\newcommand{\Z}{\mathbb{Z}}

\newcommand{\R}{\mathbb{R}}

\renewcommand{\Im}{{\operatorname{Im}\,}}

\begin{document}

\title{On the approaching geodesics property}

\author{Leandro Arosio$^1$ and Matteo Fiacchi$^2$}

 \thanks{${}^1$Partially supported by 
 PRIN Real and Complex Manifolds: Geometry and Holomorphic Dynamics n. 2022AP8HZ9, by INdAM and by the MUR Excellence Department Project MatMod@TOV
CUP:E83C23000330006.}

\thanks{${}^2$ Partially supported by Progetto FIRB - Bando 2012 ''Futuro in ricerca'' prot. RBFR12W1AQ002 ''Geometria Differenziale e Teoria Geometrica delle Funzioni'', by the European Union (ERC Advanced grant HPDR, 101053085 to Franc Forstneri\v c) and the research program P1-0291 from ARIS, Republic of Slovenia.}
\begin{abstract}
We survey some recent results and open questions on the approaching geodesics property and its application to the study of the Gromov and horofunction compactifications of a proper  geodesic Gromov metric space. We obtain results on the dynamics of isometries and  we exhibit an example of a Gromov hyperbolic domain of $\C$ which does not satisfy the approaching geodesic property.   
\end{abstract}

\maketitle
\tableofcontents

\section{Introduction}
A  proper geodesic Gromov hyperbolic metric space $X$ admits a natural compactification, called the Gromov compactification. It is a natural question to compare this compactification with the horofunction compactification. It turns out that the Gromov compactification is always ``smaller'' \cite{WWpacific} than the horofunction compactification, but in general not equivalent.  
A condition on $X$ ensuring the equivalence of the two said compactifications was highlighted in \cite{AFGG}: 
the approaching geodesic property. A metric space $X$ is said to have approaching geodesics if asymptotic rays are strongly asymptotic, that is their distance goes to 0. This is the case for instance for the unit ball of $\C^d$ endowed with the Kobayashi distance, or more generally for bounded strongly pseudoconvex domains of $\C^d$.

Rather than being just a technical condition implying the equivalence of the two compactifications, we show in this note that the approaching geodesics property has  interesting dynamical consequences which do not hold in general if we only assume the equivalence of the  two compactifications. For instance we show that if $X$ has approaching geodesics, then the minimal displacement and the divergence rate of a non-elliptic isometry coincide, and that a hyperbolic isometry always admits a unique axis. 
We also recall  results and collect several open questions concerning the approaching geodesics property and its application to complex analysis and to dynamics of non-expanding maps. Finally we
show that not all Kobayashi complete domains of $\C^d$ which are Gromov hyperbolic have approaching geodesics, 
 giving an example of a Gromov hyperbolic planar domain of $\C$ which admits a non-elliptic isometry whose minimal displacement is strictly larger than its divergence rate.

\section{Gromov and horofunction compactifications}

We recall some basic definitions. 
\begin{definition}[Compactification]\label{defcomp}
Let $X$ be a topological space. A \textit{compactification} of $X$ is a couple $(i,Y)$ where $Y$ is a compact topological space and $i\colon X\to Y$ is a topological embedding such that $\overline{i\left(X\right)}=Y$. Two compactifications  $(i_1,Y_1), (i_2,Y_2)$  are {\sl equivalent} if there exists a homeomorphism $h\colon Y_1\to Y_2$ such that $h\circ i_1=i_2$.
\end{definition}
\begin{definition}[Geodesic]
Let $(X,d)$ be a metric space.
A \textit{geodesic} is a map $\gamma$ from an interval $I\subset\R$
to $X$ which is an isometry with respect to the euclidean distance on $I$ and the distance on
$X$, that is for all $s,t\in I$, $$d(\gamma(s),\gamma(t))=|t-s|.$$ If $I=[0,+\infty)$ (resp. $\R$) we call $\gamma$ a \textit{geodesic ray} (resp. \textit{line}).
A metric space is \textit{geodesic} if any two points are joined by a geodesic. A metric space  is \textit{proper} if every closed ball
is compact.
\end{definition}
\begin{definition}[Gromov hyperbolicity]
	Let $\delta\geq0$. 
	A proper geodesic metric space $(X,d)$ is $\delta$-\textit{hyperbolic} if for  every geodesic triangle, any side is contained in the $\delta$-neighborhood of the union of the two other sides. A metric space is  \textit{Gromov hyperbolic} if it is $\delta$-hyperbolic for some $\delta$.
\end{definition}

Let $(X,d)$ be a proper geodesic Gromov hyperbolic metric space. In what follows we will study the relation between two natural compactifications of $X$: the  Gromov compactification and the horofunction compactification.

\begin{definition}[Gromov compactification]\label{visual}
	Let $(X,d)$ be a proper geodesic Gromov hyperbolic metric space. 
 Two geodesic rays $\gamma,\sigma\colon[0,+\infty)\to X$ are \textit{asymptotic} if
	$$\sup_{t\geq0}d(\gamma(t),\sigma(t))<+\infty.$$	Let $\mathscr{R}(X)$ denote the set of geodesic rays in $(X,d)$, and define the equivalence relation on $\mathscr{R}(X)$
	\[\gamma \sim_r \sigma \iff \text{$\gamma$ and $\sigma$ are asymptotic.}\]
	The \textit{Gromov boundary}  of $(X,d)$ is defined as $\partial_GX:=\mathscr{R}(X)/_{\sim_r}.$
	The \textit{Gromov compactification}  of $(X,d)$ is the set $\overline X^G:=X\sqcup \partial_GX$ endowed with a  suitable compact metrizable topology (see e.g. \cite[Chapter III.H]{BH}). 
\end{definition}

\begin{definition}[Horofunction compactification]\label{horofunctionboundary}
		Let $(X,d)$ be a proper metric space. 
		Let $C_*(X)$ be the quotient of $C(X)$ by the subspace of constant functions. 
		Given $f\in C(X)$,  we denote its equivalence class by $\overline f\in C_*(X)$. 
		
		Consider the embedding 
		$$i_H\colon X\to C_*(X)$$ which sends a point $x\in X$ to  the equivalence class of the function  $d_x\colon y\mapsto d(x,y).$ 
		The \textit{horofunction compactification} $ \overline{X}^H$ of $(X,d)$ is the closure of $i_H(X)$ in  $C_*(X)$. 
		The \textit{horofunction boundary} of $X$ is the (compact) set
		$$\partial_HX:=   \overline{X}^H\setminus i_H(X).$$
		Let $a \in \partial_H X$. An \textit{horofunction} centered at $a\in \partial_H X$ is an element $h\in C\left(X\right)$ satisfying $\bar{h}=a$. 
		For every $p\in X$, the unique horofunction centered at $a$ and vanishing at $p$ is denoted by $ h_{a,p}$.
		Let $c\in \R$. The level sets $\{ h_{a,p}<c\}$ and  $\{ h_{a,p}\leq c\}$
		are called  \textit{horospheres} (or  \textit{horoballs})   centered at $a$.	
Fix $p\in X$. Notice that a sequence $(x_n)$ in $X$ converges to $a\in  \partial_H X$ if and only if 
$$d(x_n,w)-d(w,p)\stackrel{n\to+\infty}\longrightarrow h_{a,p}(w),$$
uniformly on compact subsets.

	\end{definition}

If $(X,d)$ is a proper geodesic Gromov hyperbolic metric space, then  by \cite[Proposition 4.6]{WWpacific}  there exists a  continuous map  $\Phi\colon\overline X^H\rightarrow \overline X^G$ satisfying  $$i_G=\Phi\circ i_H,$$ where $i_{G}\colon X\rightarrow \overline X^G$ is the inclusion map.

For a general Gromov hyperbolic metric space the map $\Phi$ may not be a homeomorphism, that is, the two compactifications may not be equivalent.
\begin{example}
Let $X:=[0,\infty)\times [-1,1]$ endowed with the distance $$d((x_1,y_1),(x_2,y_2))=|x_1-x_2|+|y_1-y_2|.$$
The space $(X,d)$ is a proper geodesic Gromov hyperbolic metric space.
The Gromov boundary $\partial_GX$ consists of a single point, while the horofunction boundary $\partial_HX$ is canonically  homeomorphic to $[-1,1]$.
\end{example}
	
It is thus natural to ask for  conditions on the metric space $X$ which imply the equivalence of the two compactifications. One such condition was introduced in \cite{AFGG}.

\begin{definition}
	Let $(X,d)$ be a proper geodesic metric space.
	Two geodesic rays $\gamma,\sigma\in \mathscr{R}(X)$ are \textit{strongly asymptotic} if there exist $T\in \R$  such that
	$$\displaystyle \lim_{t\to +\infty}d(\gamma(t), \sigma(t+T))=0.$$ We say that a proper geodesic   metric space $X$ has \textit{approaching geodesics} if asymptotic rays are strongly asymptotic.
	If $(X,d)$ is a proper geodesic Gromov hyperbolic metric space and $\xi\in\partial_GX$ we say that $X$ has \textit{approaching geodesic at $\xi$} if all rays with endpoint $\xi$ are strongly asymptotic.

\end{definition}
Clearly strongly asymptotic rays are asymptotic.

\begin{remark}\label{variationasymp}(see Proposition 3.13 in \cite{AFGG})
	The rays $\gamma$ and $\sigma$ are strongly asymptotic if and only if
	\begin{equation*}
	\displaystyle \lim_{t\to+\infty}\inf_{s\geq 0}d( \gamma(t),\sigma(s))  = 0.
	\end{equation*}
\end{remark}

The next result shows that the approaching geodesic property is a sufficient condition for the equivalence of the two compactifications.
First of all, notice that every geodesic ray in $X$ has an associated horofunction.
	\begin{definition}[Busemann function]
	Let $\gamma$ be a geodesic ray in $X$. The \textit{Busemann function} $B_\gamma\colon X\times X\rightarrow \R$ associated with $\gamma$ is defined as
	$$B_\gamma(x,y):=\displaystyle \lim_{t\to +\infty} d(x,\gamma(t))-d(\gamma(t),y).$$
	For all $y\in X$, the function $x\mapsto B_\gamma(x,y)$ is a horofunction, and its class  $\overline{B}_\gamma\in \partial_HX$ does not depend on $y\in X$.
\end{definition}

\begin{remark}
Let $\gamma,\sigma$ be two strongly asymptotic geodesic rays in $X$ and let $T\in\R$ be such that 
$$\lim_{t\to+\infty}d(\gamma(t),\sigma(t+T))=0.$$
The number $T$ has an interesting interpretation in terms of the Busemann function:
$$B_\gamma(\sigma(0),\gamma(0)):=\lim_{t\to+\infty}d(\sigma(0),\gamma(t))-d(\gamma(t),\gamma(0))=\lim_{t\to+\infty}d(\sigma(0),\sigma(t+T))-d(\gamma(t),\gamma(0))=T.$$
\end{remark}

\begin{theorem}\cite[Theorem 3.5]{AFGG}\label{abstract}
	Let $(X,d)$ be a proper geodesic Gromov hyperbolic metric space with  approaching geodesics. Then the continuous map 
	$\Psi\colon  \overline{X}^G\to  \overline{X}^H$ defined by
	$$\Psi(\xi)=
	\begin{cases}
	\overline{d(\,\cdot\,,\xi)}, & {\rm if}\ \xi \in X,\\
	\overline{B}_\gamma,& {\rm if}\ \xi=[\gamma]\in \partial_GX
	\end{cases}
	$$ 
	is the inverse of $\Phi$, thus $\overline{X}^G$ and $ \overline{X}^H$ are equivalent.
\end{theorem}

\section{The approaching geodesics property in complex analysis}

In this section we present the known results and open questions on the approaching geodesics property for a bounded Kobayashi complete domain $D\subset \C^d$ and its relation with the existence of horospheres. We denote by $k_D\colon D\times D\to[0,\infty)$ the Kobayashi distance of $D$.

In 1988 Abate defined  \cite{Ab1988} the \textit{big horosphere and the small horosphere} centered at $\xi\in\partial D$ of radius $R>0$ with base point $p\in D$ respectively as
\begin{align}
\label{bigsmallhor}
\begin{split}
E^{b}_p(\xi,R)&:=\left\lbrace z\in D\colon \liminf_{w\to \xi}k_{D}(z,w)-k_{D}(w,p)<\log R\right\rbrace,\\
E^{s}_p(\xi,R)&:=\left\lbrace z\in D\colon \limsup_{w\to \xi}k_{D}(z,w)-k_{D}(w,p)<\log R\right\rbrace.
\end{split}
\end{align}

Later, using Lempert's theory of complex geodesics \cite{Lempert}, he proved that if $D$ is a bounded strongly convex domain with boundary of class $C^3$, then the limit \begin{equation}\label{abatelimit}
\lim_{w\to \xi}k_{D}(z,w)-k_{D}(w,p)
\end{equation} exists for all $\xi\in \partial D$, and thus big and small horospheres coincide (see \cite{Ab1990}). 

In the general case, it is easy to  see that the limit \eqref{abatelimit} exists for all $\xi\in\partial D$ (and so big and small horospheres coincide) if the horofunction compactification of $D$ is equivalent to the euclidean compactification 
 $\overline{D}.$
Assume now that $(D,k_D)$ is Gromov hyperbolic. In several interesting cases the Gromov compactification  $\overline{D}^G$ turns out to be equivalent to the euclidean compactification $\overline{D}$, so by Theorem \ref{abstract} the approaching geodesic property implies the existence of the limit \eqref{abatelimit} for all $\xi\in \partial D$.

Consider for example a bounded strongly pseudoconvex domain $D\subset \C^d$. Balogh and Bonk in \cite{BaBo} proved that $D$ is Gromov hyperbolic with respect to the Kobayashi distance and that $\overline{D}^G=\overline{D}$. The approaching geodesic property for bounded strongly pseudoconvex domains was established in \cite{AFGG} via rescaling with the squeezing function.

\begin{definition}[Squeezing function]
	\label{DefSq}
	Let $D\subseteq \C^d$ be a bounded domain. Given $z\in D$  define  the family of functions
	\begin{alignat*}{2}
	\mathcal{F}_z&:= \{\varphi : D\to \mathbb{B}^d \; | \text{ $\varphi$ is holomorphic, injective and } \varphi(z)=0\}.
	\end{alignat*}
	The \textit{squeezing function} of $D$ at the point $z \in D$ is defined as  
	$$
	s_D(z):=\sup\{r>0 :  B(0,r) \subseteq \varphi(D) \text{ for some }\varphi\in\mathcal F_z\}.
	$$
\end{definition}
In \cite[Theorem 1.3]{DGZ} F. Deng, Q. Guan, and L. Zhang proved that if $D$ is a bounded strongly pseudoconvex domain with boundary of class $C^2$ then $\displaystyle \lim_{z \to \partial D}s_D(z)= 1$.

\begin{theorem}\label{raysgetstogether}\cite[Theorem 4.3]{AFGG}
	Let $D\subseteq \C^d$ be a domain satisfying $\displaystyle \lim_{z\to \partial  D} s_D= 1$. 
	Then $(D,k_D)$ has approaching geodesics.
\end{theorem}

The idea of the proof is the following: if by contradiction there exist two geodetic  rays $\gamma_1,\gamma_2$ in $D$ which are asymptotic but not strongly asymptotic then rescaling with the univalent maps $\varphi\colon D\to\B^d$ provided by the squeezing function along $\gamma_1$ we end up in the limit with two distinct geodesic lines $ \hat\gamma_1,\hat\gamma_2\colon \R\to \B^d$ of the unit ball which are asymptotic both at $+\infty$ and $-\infty$, which is a contradiction.

\begin{corollary}\cite[Corollary 4.5]{AFGG}\label{ohyeah1.5.1}
	Let $D\subseteq \C^d$ be a bounded strongly pseudoconvex domain with $C^2$ boundary.
	Then $(D,k_D)$ has approaching geodesic. Moreover, since $\overline{D}^G=\overline{D}$, the limit \eqref{abatelimit} exists and thus big and small horospheres coincide.
\end{corollary}

Other Gromov hyperbolic domains for which it is known that the euclidean compactification is equivalent to the Gromov compactification include   bounded pseudoconvex finite type domains in $\C^2$  \cite{fia} and Gromov hyperbolic bounded convex domains \cite{BrGaZi}. The following two open questions are thus natural.

\begin{question}
	Let $D\subset\C^2$ be a smooth pseudoconvex finite type domain,  Does $(D,k_D)$ have approaching geodesic? Are the horofunction and Gromov compactification equivalent?
\end{question}
\begin{question}
	Let $D\subset\C^d$ be a Gromov hyperbolic bounded convex domain.  Does $(D,k_D)$ have approaching geodesic? Are the horofunction and Gromov compactification equivalent?
\end{question}
In the latter question, if the domain $D\subset\C^d$ is smoothly bounded then by \cite{Zim1} it  has finite type.
It follows that  the horofunction and Gromov compactification of $D$ are equivalent, since they satisfy the (weaker) approaching {\sf complex} geodesic property which we now introduce.
Recall that a complex geodesic is a holomorphic map $\varphi\colon\D\to D$ that is an isometry with respect to the Kobayashi distances. In particular, if $\varphi$ is a complex geodesic then the curve $\gamma\colon [0,+\infty)\to D$ given by
\begin{equation}\label{parCgeo}
\gamma(t)=\varphi(\tanh(t\slash2))
\end{equation}
is a geodesic ray of $(D,k_D)$.

\begin{definition}
Let $D\subset \C^d$ be a bounded convex domain. 
We call {\sl complex geodesic ray} $\gamma\colon[0,+\infty)\to D$ any geodesic ray which is contained in a complex geodesic, or equivalently the curve defined in (\ref{parCgeo})
where $\varphi\colon \D\to D$ is a complex geodesic.
We say that $D$ has {\sl approaching complex geodesics} if any two complex geodesic rays which are asymptotic are strongly asymptotic.
\end{definition}

The proof of Theorem \ref{abstract} can be adapted to show the following.
\begin{theorem}\cite[Theorem 3.11]{AFGG}
Let $D\subset  \C^d$ be a Gromov hyperbolic bounded convex domain. Assume that $D$ has approaching complex geodesics. Then $\overline{D}^G$ and $ \overline{D}^H$ are equivalent.
\end{theorem}

A scaling argument similar to the one in the proof of Theorem \ref{raysgetstogether}, using  McNeal's scaling instead of the squeezing function yields the following result.
\begin{theorem}\cite[Proposition 5.4]{AFGG}\label{asyfintyp}
	Let  $D\subseteq \C^d$ be a bounded convex finite type domain. Then $D$ has approaching complex geodesics, and thus  $\overline{D}^G$ and $ \overline{D}^H$ are equivalent.
	 \end{theorem}

The following weaker version of Question 2  is also open.
\begin{question}
	Let $D\subset\C^d$ be a bounded convex domain such that $(D,k_D)$ is Gromov hyperbolic. Does $(D,k_D)$ have approaching complex geodesics?  
\end{question}

\section{Dynamics of non-expanding self-maps}

In this section we discuss the dynamics  of non-expanding self-maps of proper geodesic Gromov hyperbolic spaces (see \cite{AFGG} and \cite{AFGK}). In particular we show that the approaching geodesic property implies  dynamical results, like the equality between the minimal displacement and the divergence rate of isometries, or the existence and uniqueness of an axis for hyperbolic isometries.

In \cite[Section 6]{AFGG} several classical concepts from complex dynamics were generalized to  this metric setting as follows.
	\begin{definition}[Dilation]\label{defdilation}
		Let $(X,d)$ be a proper geodesic Gromov hyperbolic metric space. 
		Let $f\colon X\to X$ be a non-expanding self-map. 
		Given $\eta\in \partial_GX$, the {\sl dilation} of $f$ at $\eta$ with respect to the  base point $p\in X$ is the number $\lambda_{\eta,p}>0$ defined by
		$$\log\lambda_{\eta,p}=\liminf_{z\to \eta} d(z,p)-d(f(z),p).$$ 
	\end{definition}
	
\begin{definition}[Geodesic region]\label{georegion}
	Let $(X,d)$ be a proper geodesic Gromov hyperbolic metric space. 
	Given $R>0$ and  a geodesic ray $\gamma\in \mathscr{R}(X)$, the \textit{geodesic region} $A(\gamma,R)$ is the open subset of $X$ of the form
	$$A(\gamma, R):=\{x\in X\colon d(x,\gamma)<R\}.$$
	The point $[\gamma]\in\partial_GX$ is called the \textit{vertex} of the geodesic region.
\end{definition}

\begin{definition}[Geodesic limit]
	Let $(X,d)$ be a proper geodesic Gromov hyperbolic metric space and let $Y$ be a Hausdorff topological space. Let $f\colon X\rightarrow Y$ a map, and  let $\eta\in\partial_GX$, $\xi\in Y$. We say that $f$ has \textit{geodesic limit} $\xi$ at $\eta$
	if for every sequence $(x_n)$ converging to $\eta$ contained in a geodesic region with vertex $\eta$, the sequence $(f(x_n))$ converges to $\xi$.
\end{definition}	
	
\begin{definition}[Boundary regular fixed points]\label{defbrfp}
	Let $(X,d)$ be a proper geodesic Gromov hyperbolic metric space. Let $f\colon X\to X$ be a non-expanding map. We say that a point $\eta\in \partial_GX$ is  a {\sl boundary regular fixed point} (BRFP for short) if
	$\lambda_{\eta,p}<+\infty$ and  if $f$ has geodesic limit $\eta$ at $\eta$. 
	
\end{definition}

In general the dilation at a BRFP may depend on the chosen base-point $p$. This is not the case if the horofunction and Gromov compactifications are equivalent (see \cite[Proposition 6.30]{AFGG}).
Hence, from now on, if $\overline X^H$ is topologically equivalent to $\overline X^G$ and $\eta\in\partial_GX$ is a BRFP we denote with $\lambda_{\eta}$ the dilation $\lambda_{\eta,p}$.
In this case the dilation $\lambda_{\eta}$ can be conveniently calculated along geodesic rays, as the following result shows.
\begin{lemma}\cite[Lemma 3.12]{AFGK}\label{prop:JFCGeneral}
	Let $(X,d)$ be a proper geodesic Gromov hyperbolic metric space such that $\overline X^H$ is  equivalent to $\overline X^G$. 
	Let $f\colon X\to X$ be a non-expanding self-map, 
	let $\eta\in \partial_GX$ be a BRFP and let $\gamma\colon[0,+\infty)\rightarrow X$ be a geodesic with $\gamma(+\infty)=\eta$. Then
	\begin{equation}
	\label{limitgeo}
	\lim_{t\to+\infty} d(p,\gamma(t)) - d(p,f(\gamma(t))) = \log \lambda_{\eta}.
	\end{equation} 
\end{lemma}

We now recall the definition of two numbers that play an important role  in the study of the  dynamics of non-expanding maps.

\begin{definition}
	Let $(X,d)$ be a metric space and $f\colon X\to X$ be a non-expanding self-map. 
	\begin{itemize}
		\item Let $x\in X$, the {\sl divergence rate} (or {\sl escape rate}, or {\sl translation length}) $c(f)$ of $f$ is the limit
		$$c(f):=\lim_{n\to+\infty}\frac{d(x,f^n(x))}{n}. $$
		The limit exists thanks to the subadditivity of the sequence $d(x,f^n(x))$, and  does not depend on the choice of $x\in X$. Clearly for all $n>0$ we have $c(f^n)=nc(f)$.
		\item The \textit{minimal displacement} $\tau(f)$ of $f$ is defined as $$\tau(f):=\inf_{x\in X} d(x,f(x)).$$ Since
		$d(x,f^n(x))\leq nd(x,f(x))$ we have that $c(f)\leq \tau(f)$ and $\tau(f^n)\leq n\tau(f)$.
	\end{itemize}
\end{definition}

Using the divergence rate we can classify non-expanding self-maps of $X$ as follows. We say that $f$ is {\sl elliptic} it it has a bounded orbit (or equivalently by Calka's theorem \cite{Calka} if every orbit is bounded). If $f$ is not elliptic, we say that it is {\sl parabolic} if $c(f)=0$ and that it is {\sl hyperbolic} if $c(f)>0$.

\begin{remark}
	Let $(X,d)$ be a metric space and $f\colon X\to X$ be a non-expanding self-map. Then
	$$\lim_{n\to+\infty}\frac{\tau(f^n)}{n}=c(f).$$
	Indeed for all $n>0$,
	$$c(f)=\frac{c(f^n)}{n}\leq\frac{\tau(f^n)}{n}$$
	and if $x\in X$
	$$\frac{\tau(f^n)}{n}\leq\frac{d(x,f^n(x))}{n}\to c(f).$$
\end{remark}

A natural problem is to characterize when $c(f)=\tau(f)$, see e.g. \cite[Theorem 1]{GV}, where such equality is proved under the assumption that $X$ is complete and satisfies a mild form of Busemann's classical non-positive curvature condition. 
The situation in the elliptic case is trivial: if $f$ is elliptic then $c(f)=0$ but $\tau(f)=0$ if and only if $f$ has a fixed point. Hence we will focus on the non-elliptic case.
We first recall two results about the Denjoy--Wolff point of a non-elliptic map.
\begin{theorem}\cite{Kar}\label{DW}
Let $(X,d)$ be a proper Gromov hyperbolic metric space. Let $f\colon X\to X$ be a  non-elliptic non-expanding   self-map. Then there exists a unique $\zeta\in\partial_GX$, called the \textit{Denjoy--Wolff point} of $f$, so that for all $x\in X$ we have
	$$\lim_{n\to+\infty}f^n(x)=\zeta.$$
\end{theorem}

\begin{theorem}\cite[Theorem 6.32]{AFGG}\label{Korhyp}
	Let $(X,d)$ be a proper geodesic Gromov hyperbolic metric space such that $\overline X^H$ is topologically equivalent to $\overline X^G$. Let $f\colon X\to X$ be a non-elliptic non-expanding self-map.
	Let $\zeta\in \partial_GX$ be its Denjoy--Wolff point. Then $\zeta$ is the unique BRFP of $f$ with dilation $\lambda_{\zeta}\leq1$. Moreover
	\begin{equation}\label{divergencerate}
	c(f)=-\log\lambda_\zeta.
	\end{equation}
\end{theorem}

We now introduce a condition on $f$ implying the equality of divergence rate and minimal displacement.
\begin{proposition}\label{tcprop}
	Let $(X,d)$ be a proper geodesic Gromov hyperbolic metric space such that $\overline X^H$ is topologically equivalent to $\overline X^G$.
	Let $f\colon  X\rightarrow X$ be a non-expanding map, let $\eta\in\partial_GX$ be a BRFP, and  let $\gamma\colon [0,+\infty)\rightarrow X$ be a geodesic with $\gamma(+\infty)=\eta$.
	Suppose that 
	\begin{equation}\label{condition}\lim_{t\to+\infty}\inf _{s\geq0}d(f(\gamma(t)),\gamma(s))=0,
	\end{equation} then
\begin{equation}\label{JWC}
	\lim_{t\to+\infty} d\big(f(\gamma(t))\,,\, \gamma(t-\log\lambda_\eta)\big) =0.
\end{equation}
In particular if the map is non-elliptic then $\tau(f)=c(f)$.
\end{proposition}
\proof
Let $s_t\ge 0$ such that $d(f(\gamma(t)),\gamma)=d(f(\gamma(t)), \gamma(s_t))\to0$.
By Proposition \ref{prop:JFCGeneral} we have
\begin{align*}
\log\lambda_\eta &= \lim_{t\to+\infty} d(\gamma(t), \gamma(0)) - d(f(\gamma(t)), \gamma(0))\\
&= \lim_{t\to+\infty} d(\gamma(t), \gamma(0)) - d(\gamma(s_t), \gamma(0))\\
&= \lim_{t\to+\infty} t - s_t.
\end{align*}
So,\begin{align*}
d\big(f(\gamma(t))\,,\, \gamma(t-\log\lambda_\eta)\big)&\leq d(f(\gamma(t)), \gamma(s_t))+d(\gamma(s_t),\gamma(t-\log\lambda_\eta))\\&=d(f(\gamma(t)), \gamma(s_t))+|t-s_t-\log\lambda_\eta|\stackrel{t\to+\infty}\longrightarrow 0
\end{align*}
In particular if the map is not elliptic, applying (\ref{JWC}) and (\ref{divergencerate}) to the Denjoy-Wolff point $\zeta\in\partial_GX$ we obtain,
$$\tau(f)\leq \lim_{t\to\infty}d(f(\gamma(t)),\gamma(t))\leq \lim_{t\to\infty}d\big(f(\gamma(t))\,,\, \gamma(t-\log\lambda_\zeta)\big)+d\big(\gamma(t-\log\lambda_\zeta),\gamma(t)\big)=-\log\lambda_\zeta=c(f). $$
\endproof

By Remark \ref{variationasymp},   condition \eqref{condition}  is satisfied in the following interesting case.
\begin{corollary}\label{ctiso}
	Let $(X,d)$ be a proper geodesic Gromov hyperbolic metric space such that $\overline X^H$ is topologically equivalent to $\overline X^G$. Let $f\colon  X\rightarrow X$ be a non-elliptic isometry, then $$\tau(f)=c(f).$$
\end{corollary}

The approaching geodesics condition is crucial in the previous result. Indeed,  if we only assume the equivalence of the horofunction and Gromov compactifications, then the minimal displacement can be strictly larger than the divergence rate, as the following example shows.

\begin{example}
	Consider the infinite cylinder $X:=\{(x,y,z)\in \R^3\colon y^2+z^2=1\}$ with the Riemannian metric inherited from the Euclidean metric on $\R^3$. If $d$ is the associated distance, we have that $\overline X^H$ is equivalent to $\overline X^G$.
	Consider the isometry
	$$f(x,y,z)=(x+1,-y, -z).$$ Then $c(f)=1$ but $\tau(f)=\sqrt{1+\pi^2}$.
\end{example}

\begin{question}\label{sally}
	Let $(X,d)$ be a proper geodesic Gromov hyperbolic metric space with approaching geodesics.
	Let $f\colon X\to X$ be a non-expanding non-elliptic self-map. Is it true that $c(f)=\tau(f)$?
\end{question}	
We now discuss a possible strategy to answer the previous question. 

Let $(X,d)$ be a proper geodesic Gromov hyperbolic metric space such that $\overline X^H$ is  equivalent to $\overline X^G$. Let $f\colon X\to X$ be a non-expanding self-map, let  $\eta\in \partial_GX$ be a BRFP and let $\gamma\colon [0,+\infty)\rightarrow X$ be a geodesic with $\gamma(+\infty)=\eta$. Then it follows from 
 \cite[Proposition 3.6]{AFGK} that the distance between $\gamma(t)$ and $f(\gamma(t))$ is bounded by a constant independent on $t$.
Moreover, by \cite[Lemma 3.13]{AFGK} the curve $f\circ \gamma$ is a quasi-geodesic of the following very special type.

\begin{definition}
	Let $(X,d)$ be a geodesic metric space. We say that $\gamma\colon [0,+\infty)\rightarrow X$ is an \textit{almost-geodesic} if for each $\epsilon>0$ there exists $t_\epsilon\geq0$ such that for all $t_1,t_2\geq t_\epsilon$
	$$|t_1-t_2|-\epsilon\leq d(\gamma(t_1),\gamma(t_2))\leq |t_1-t_2|.$$
\end{definition}

By Proposition \ref{tcprop} a positive answer to Question \ref{sally} would
follow if one could show that the approaching geodesics property implies the following stronger property involving almost-geodesics.

\begin{question}
	Let $(X,d)$ be a proper geodesic Gromov hyperbolic metric space with approaching geodesics. Let  $\gamma$ be a geodesic ray and $\sigma$ an almost-geodesic such that $$\sup_{t\geq0}d(\gamma(t),\sigma(t))<M.$$ Is it true that $$\lim_{t\to+\infty}\inf_{s\geq0}d(\gamma(s),\sigma(t))=0?$$
\end{question}

Another concept related to the divergence rate and minimal displacement is the axis of an isometry.
\begin{definition}\label{axes1}
	Let $(X,d)$ be a metric space and $f\colon X\to X$ be an isometry without fixed points. An  \textit{axis} is a setwise invariant geodesic line $\gamma$.
	An isometry without fixed points is \textit{axial} if there exists an axis.
\end{definition}
\begin{remark} Let $\gamma$ be an axis.
	Up to changing the orientation of the geodesic line $\gamma\colon \R\to X$ if needed, there exists $a>0$ such that
	$$f(\gamma(t))=\gamma(t+a), \ \ \ \forall t\in\R.$$
\end{remark}

\begin{remark} If $(X,d)$ is a proper geodesic Gromov hyperbolic metric space, then an axial isometry $f$ is always hyperbolic (i.e. $c(f)>0$) and an axis $\gamma$ is a geodesic line connecting the Denjoy-Wolff point of $f^{-1}$ to the Denjoy-Wolff point of $f$. Moreover if $a>0$ is  such that $f(\gamma(t))=\gamma(t+a)$ then $a=c(f)$, indeed
	$$a=\frac{d(\gamma(na),\gamma(0))}{n}=\frac{d(f^n(\gamma(0)),\gamma(0))}{n}\stackrel{n\to+\infty}\longrightarrow c(f).$$
	
\end{remark}

In the next Lemma we show that  the minimal displacement plays a relevant role in the construction of axes. With $\lfloor\cdot\rfloor$ we indicate the floor function.

\begin{lemma}\label{lemmaaxis}
Let $(X,d)$ be a metric space and $f\colon X\to X$ be an isometry with $\tau(f)=c(f)>0$. Assume that the minimal displacement is attained, that is there exists $x_0\in X$ such that $d(x_0,f(x_0))=\tau(f)$. Then $x_0$ is contained in an axis.
\end{lemma}
\proof 
First of all for all $n\in\Z$ we have
$$\tau(f^n)\leq |n|\tau(f)=|n|c(f)=c(f^n)$$
which implies $\tau(f^n)=|n|\tau(f)$.
Notice that up to a rescaling of the distance we can suppose $\tau(f)=c(f)=1$.
Let
$\gamma\colon [0,1]\to X$ be a geodesic with $\gamma(0)=x_0$ and $\gamma(1)=f(x_0)$. We claim that  the curve
$\Gamma\colon \R\to X$ given by
$$\Gamma(t)=f^{\lfloor t\rfloor}(\gamma(t- \lfloor t\rfloor)), \ \ \ t\in\R$$
is an axis for $f$.
Clearly $\Gamma$ is setwise invariant curve, it remains to show that it is a geodesic line.
Fix $s\leq t$. If $\lfloor s\rfloor=\lfloor t\rfloor$, since $\gamma$ is a geodesic segment we have
$$d(\Gamma(s),\Gamma(t))=d(\gamma(s-\lfloor s\rfloor),\gamma(t-\lfloor t\rfloor))=t-s.$$
Now assume  $\lfloor s\rfloor<\lfloor t\rfloor$. Then
$$d(\Gamma(s),\Gamma(t))\leq d(\Gamma(s),\Gamma(\lfloor s\rfloor+1))+d(\Gamma(\lfloor s\rfloor+1),\Gamma(\lfloor t\rfloor))+d(\Gamma(\lfloor t\rfloor),\Gamma(t))=t-s.$$
Now since $\tau(f^n)=|n|$, we have for all $n,m\in\Z$ 
$$d(\Gamma(n),\Gamma(m))=d(f^n(x_0),f^m(x_0))=d(x_0,f^{m-n}(x_0))\geq \tau(f^{m-n})=|m-n|,$$ hence
$$d(\Gamma(s),\Gamma(t))\geq d(\Gamma(\lfloor s\rfloor),\Gamma(\lfloor t\rfloor+1))-d(\Gamma(s),\Gamma(\lfloor s\rfloor))-d(\Gamma(\lfloor t\rfloor+1),\Gamma(t))=t-s.$$
\endproof

\begin{proposition}\label{attained}
Let $(X,d)$ be a proper geodesic Gromov hyperbolic metric space and let $f\colon X\to X$ be a hyperbolic isometry. Then the minimal displacement is attained.
\end{proposition}
\begin{proof}
 Let $(x_\nu)$ be such that $$\tau_\nu:=d(x^\nu,f(x^\nu))\searrow\tau(f)$$
and denote  $x^\nu_n:=f^n(x^\nu)$ for all $n\in\Z$.
For a fixed  $\nu$, we have for all $n,m\in\Z$
$$d(x^\nu_n,x^\nu_m)=d(f^n(x^\nu),f^m(x^\nu))=d(f^{n-m}(x^\nu),x^\nu)\leq |n-m|d(x^\nu,f(x^\nu))=|n-m|\tau_\nu.$$
On the other hand
$$d(f^{n-m}(x^\nu),x^\nu)\geq c(f^{n-m})=|n-m|c(f),$$
in other words the sequences $(x^\nu_n)_n$ are discrete quasi geodesic with uniformly bounded constants joining the two boundary fixed points of $f$, so by the visibility property (see \cite[Chapter 5]{GdlH}) there exists $n_\nu\in\Z$ such that $f^{n_\nu}(x^\nu)$ converges to $x_0\in X$ with the property that $d(x_0,f(x_0))=\tau(f)$. 

\end{proof}

\begin{corollary}\label{ctaxial}
	Let $(X,d)$ be a proper geodesic Gromov hyperbolic metric space and let $f\colon X\to X$ be a hyperbolic isometry. Then $f$ is axial if and only if $c(f)=\tau(f)$.
\end{corollary}
\proof
 If $\gamma$ is an axis, then 
$$c(f)=d(\gamma(c(f)),\gamma(0))=d(f(\gamma(0)),\gamma(0))\geq \tau(f).$$
The converse implication follows from Lemma \ref{lemmaaxis} and Proposition \ref{attained}.
\endproof

\begin{proposition}\label{approaxis}
	Let $(X,d)$ be a proper geodesic Gromov hyperbolic metric space and $f\colon X\to X$ be a hyperbolic isometry with Denjoy-Wolff point $\zeta\in\partial_GX$. 
	If $(X,d)$ has approaching geodesics at $\zeta$ then $f$ is axial and the axis is unique.
\end{proposition}
\proof
By Corollary \ref{ctiso} and Corollary \ref{ctaxial}, $f$ is axial. Let $\gamma_1$ and $\gamma_2$ be two axes. Since the metric space $(X,d)$ has approaching geodesics at $\zeta$, it follows that, up to a change of parametrization, $$\lim_{t\to+\infty}d(\gamma_1(t),\gamma_2(t))=0$$
and for $i=1,2,$
$$f(\gamma_i(t))=\gamma_i(t+c(f)), \ \ \ t\in\R.$$
Finally, for all $t_0\in\R$ we have
$$d(\gamma_1(t_0),\gamma_2(t_0))=d(f^n(\gamma_1(t_0)),f^n(\gamma_2(t_0)))=d(\gamma_1(t_0+nc(f))),\gamma_2(t_0+nc(f)))\stackrel{n\to+\infty}{\longrightarrow}0,$$
and thus  $\gamma_1=\gamma_2$.
\endproof

\section{A Gromov hyperbolic planar domain without approaching geodesics}

In this section we discuss an example of a  hyperbolic domain $D\subset \C$ that is Gromov hyperbolic with respect to the Kobayashi  distance but does not satisfy the approaching geodesics property. Indeed we will construct a hyperbolic isometry with $\tau(f)<c(f)$, which is in contrast with Corollary \ref{ctiso}.

Denote by  $i(z)=\bar{z}$  the complex conjugation, and by $h(z)=z+1$  the translation by 1.
Let  $C\subset  \{z\in\C: |\Im z|<1\}$ be a closed set such that 
\begin{enumerate}
\item $i(C)=C$ and $h(C)=C$;
\item $C\cap \R\neq \emptyset$;
\item $D:=  \{z\in\C: |\Im z|<1\}\setminus C$ is connected.
\end{enumerate}
For instance, take $C=\Z$.
Let $f\colon D\to D$ given by
$$f(z)=\bar{z}+1.$$
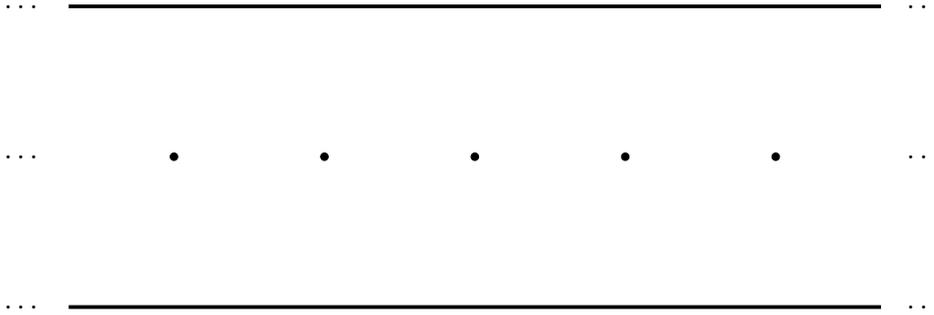
\begin{figure}[ht]
	\begin{tikzpicture}[scale=2,  mydot/.style={circle, fill=black, draw, outer sep=0pt,inner sep=1pt}]
	\draw[line width=.5mm] (-0.7,1)--(4.7,1);
	\draw[line width=.5mm] (-0.7,-1)--(4.7,-1);
	
	\node at (5,1) {$\dots$};
	\node at (5,0) {$\dots$};
	\node at (5,-1) {$\dots$};
	\node at (-1,1) {$\dots$};
	\node at (-1,0) {$\dots$};
	\node at (-1,-1) {$\dots$};

	\node[mydot]  at (0,0) {};
	\node[below right] at (0,0) {};
	\node[mydot]  at (1,0) {};
	\node[below right] at (1,0) {};
	\node[mydot]  at (2,0) {};
	\node[below right] at (2,0) {};
	\node[mydot]  at (3,0) {};
	\node[below right] at (3,0) {};
	\node[mydot]  at (4,0) {};
	\node[below right] at (4,0) {};

	\end{tikzpicture}
	\caption{The domain $D$ if $C=\Z$}
\end{figure}

 It follows from   \cite[Theorem 3.2]{PorRodriTou} that 
the metric space $(D,k_{D})$ is Gromov hyperbolic.
\begin{lemma}
The map $f$ is a hyperbolic isometry.
\end{lemma}
\proof
Since $i(D)=D$, $f$ is an isometry of $D$ because it is an antiholomorphic bijection.

Denote with $S:=\{z\in\C: |\Im z|<1\}$ the strip. Notice that $f^2$ is the holomorphic map $z\mapsto z+2$ which extends holomorphically to a hyperbolic automorphism $g$ of the strip $S$. Since $D\subset S$, using the decreasing property of the Kobayashi distance we have
$$c(f)=\frac{c(f^2)}{2}\geq \frac{c(g)}{2}>0.$$
\endproof

\begin{proposition}
The metric space $(D,k_D)$ does not satisfy the approaching geodesics property.
\end{proposition}

\proof
 Assume, to get a contradiction, that $(D,k_D)$ satisfies the approaching geodesics property. Then by Corollary  \ref{ctiso} it follows that $c(f)=\tau(f).$ Then by Proposition \ref{ctaxial} there exists an axis $\gamma$, which satisfies 
$$f(\gamma(t))=\gamma(t+c(f)).$$ There exists $t_0\in\R$ such that $x_0:=\gamma(t_0)\in\R$ (we can suppose that $t_0=0$). 
The map $h(z)=z+1$ is a  hyperbolic isometry of $D$. Notice that $h^2=f^2$, which implies
$$c(h)=\frac{c(h^2)}{2}=\frac{c(f^2)}{2}=c(f)$$ and $f(x_0)=h(x_0)$, so
$$\tau(h)\leq k_D(x_0,h(x_0)))=k_D(x_0,f(x_0))=\tau(f) $$
hence $c(h)=\tau(h)=:\tau$.
Now by the proof of Lemma \ref{lemmaaxis} the curve
$$\sigma(t)=h^{\lfloor t\slash\tau\rfloor}(\gamma(t-\tau\lfloor t\slash\tau\rfloor)), \ \ \ t\in\R,$$
is a geodesic line (an axis for $h$). Since the  distance   $k_D$ is induced by a Riemannian metric on $D$,  and $\gamma$ and $\sigma$ are two geodesic lines with $\sigma|_{[0,\tau]}=\gamma|_{[0,\tau]}$, it follows that  $\gamma=\sigma$. In particular $\gamma(t)=\sigma(t)\in\R$ for $t\in [\tau,2\tau]$ but this is impossible because $D\cap [x_0+1,x_0+2]$ is disconnected.
\endproof

We conclude with an open question concerning the  domain $D$.
\begin{question}
	For the metric space $(D,k_{D})$ are the horofunction and Gromov compactification equivalent?
\end{question}

{On behalf of all authors, the corresponding author states that there is no conflict of interest.}


\begin{thebibliography}{88} 
	
\bibitem{Ab1988} M. Abate, {\sl Horospheres and iterates of holomorphic maps}, Math. Z. {\bf 198}(2) (1988),  225--238.
	
\bibitem{Ab1990} M. Abate, {\sl The Lindel\"of principle and the angular derivative in strongly convex domains}, J. Analyse Math. {\bf 54} (1990), 189--228.	
	
	\bibitem{AFGG} L. Arosio, M. Fiacchi, S. Gontard, L. Guerini, {\sl The horofunction boundary of a Gromov hyperbolic space}, Math. Ann. \textbf{388}(2) (2024), 1163--1204.
	
	\bibitem{AFGK} 	L. Arosio, M. Fiacchi, L. Guerini, A. Karlsson {\sl Backward dynamics of non-expanding maps in Gromov hyperbolic metric spaces}, Adv. in Math. \textbf{439} (2024), 109484.

	\bibitem{BaBo} Z. M. Balogh, M. Bonk, {\sl Gromov hyperbolicity and the Kobayashi metric on strictly pseudoconvex domains}, Comment. Math. Helv. {\bf 75} (2000), 504--533.

\bibitem{BrGaZi} F. Bracci, H. Gaussier, A. Zimmer, {\sl Homeomorphic extension of quasi-isometries for convex domains in $\mathbb C^d$ and iteration theory}, Math. Ann. {\bf 379} (2021), 691--718.

 	\bibitem{BH} M. Bridson, A. Haefliger, {\sl Metric spaces of nonpositive curvature}, Grundlehren der Mathematischen Wissenschaften [Fundamental Principles of Mathematical Sciences] {\bf 319}, Springer-Verlag, Berlin (1999).
	
\bibitem{Calka} A. Calka, {\sl On conditions under which isometries have bounded orbits}, Colloq. Math. {\bf 48}(2) (1984), 219--227.	
	
	 \bibitem{CDP} M. Coornaert, T. Delzant, A. Papadopoulos, {\sl G\'eom\'etrie et th\'eorie des groupes}, Lecture Notes in Mathematics {\bf 1441} (1990).
 
 \bibitem{DGZ} F. Deng, Q. Guan and L. Zhang, {\sl Properties of squeezing functions and global transformations of bounded domains}, Trans. Amer. Math. Soc. {\bf 368}(4) (2016),  2679--2696.
 
 
 
 
 \bibitem{fia}
 M. Fiacchi, {\sl Gromov hyperbolicity of pseudoconvex finite type domains in $\C^2$}, Math. Ann. {\bf 382}(1) (2022), 37--68.
 
 \bibitem{GV}
 S. Gaubert, G. Vigeral, {\sl A maxmin characterisation of the escape rate of non-expansive mappings in metrically convex spaces}, Math. Proc. Cambridge Philos. Soc. {\bf 152}(2) (2012), 341--363.
 
 \bibitem{GdlH} \'E. Ghys, P. De La Harpe,  {\sl Sur les groupes hyperboliques, d'apr\`es Mikhael Gromov}, Progress in Mathematics {\bf 83}, Birkh\"{a}user Boston, Inc., Boston, MA (1990). 
 
 \bibitem{Kar} A. Karlsson, {\sl Non-expanding maps and Busemann functions}, Ergodic Theory Dynam. Systems {\bf 21}(5) (2001),  1447--1457.
 
 \bibitem{Lempert} L. Lempert, {\sl La m\'etrique de Kobayashi et la repr\'esentation des domaines sur la boule}, Bull. Soc. Math. France {\bf 109}(4) (1981),  427--474.
	
	\bibitem{PorRodriTou} A. Portilla, J.M. Rodr\'iguez, E Tour\'is, {\sl The topology of balls and Gromov hyperbolicity of Riemann surfaces},	Differential Geometry and its Applications \textbf{21}(3) (2004), 317--335. 
	
	\bibitem{WWpacific}
	C. Webster,  A. Winchester, {\sl Boundaries of hyperbolic metric spaces}, Pacific J. Math. {\bf 221}(1) (2005),  147--158.
	
	\bibitem{Zim1} A. Zimmer, {\sl Gromov hyperbolicity and the Kobayashi metric on convex domains of finite type}, Math. Ann. {\bf 365}(3--4) (2016), 1425--1498.
	
\end{thebibliography}
\end{document}